\documentclass[11pt]{amsart}
\usepackage{tabularx,booktabs}
\usepackage{caption}
\usepackage{amsmath}
\usepackage{amsfonts}

\usepackage{amscd}
\usepackage{amsthm}
\usepackage{amssymb} \usepackage{latexsym}
\usepackage{eufrak}
\usepackage{euscript}
\usepackage{epsfig}
\usepackage{graphics}
\usepackage{array}
\usepackage{enumerate}
\usepackage{dsfont}
\usepackage{color}
\usepackage{wasysym}
\usepackage{hyperref}
\usepackage{pdfsync}

\newcommand{\bel}[1]{\begin{equation}\label{#1}}

\newcommand{\be}{\begin{equation}}

\newcommand{\ba}{\begin{eqnarray}}
\newcommand{\ea}{\end{eqnarray}}

\newcommand{\qe}{\end{equation}}

\newcommand{\Deg}{\mathrm{Deg}}

\newcommand{\Ric}{{\textrm Ric}}

\newcommand{\Hmm}[1]{\leavevmode{\marginpar{\tiny%
$\hbox to 0mm{\hspace*{-0.5mm}$\leftarrow$\hss}%
\vcenter{\vrule depth 0.1mm height 0.1mm width \the\marginparwidth}%
\hbox to
0mm{\hss$\rightarrow$\hspace*{-0.5mm}}$\\\relax\raggedright #1}}}

\newtheorem{theorem}{Theorem}[section]

\newtheorem{corollary}[theorem]{Corollary}

\newtheorem{remark}[theorem]{Remark}

\newtheorem{prop}[theorem]{Proposition}

\begin{document}

\title[Time analyticity of ancient solutions to the heat equation on graphs]{Time analyticity of ancient solutions to the heat equation on graphs}

\author{Fengwen Han}
\email{\href{mailto:19110180022@fudan.edu.cn}{19110180022@fudan.edu.cn}}
\address{Fengwen, Han: School of Mathematical Sciences, Fudan University, Shanghai 200433, China}

\author{Bobo Hua}
\email{\href{mailto:bobohua@fudan.edu.cn}{bobohua@fudan.edu.cn}}
\address{Bobo Hua: School of Mathematical Sciences, LMNS, Fudan University, Shanghai 200433, China; Shanghai Center for Mathematical Sciences, Fudan University, Shanghai 200433, China}

\author{Lili Wang}
\email{\href{mailto:lili\_wang@fudan.edu.cn}{lili\_wang@fudan.edu.cn}}
\address{Lili Wang: School of Mathematical Sciences, Fudan University, Shanghai 200433, China}

\begin{abstract}
We study the time analyticity of ancient solutions to heat equations on graphs. Analogous to Dong and Zhang \cite{DongZhang19}, we prove the time analyticity of ancient solutions on graphs under some sharp growth condition.
\end{abstract}
\maketitle 

\par
\maketitle

\bigskip

\section{Introduction}\label{sec:intro}
 
In the investigation of global solutions to elliptic equations, the well-known Liouville theorem states that any bounded (or positive) harmonic function on $\mathbb R^n$ is constant. This has been generalized to Riemmannian manifolds with nonnegative Ricci curvature by Yau \cite{Yau75}. For evolution equations, ancient solutions, i.e. solutions of the heat equation defined on the whole space and all negative time, are natural generalizations of harmonic functions. Souplet and Zhang \cite{SoupletZhang06} proved that any sublinear ancient solution is constant for a Riemannian manifold with nonnegative Ricci curvature. Later, for ancient solutions of polynomial growth, Lin and Zhang \cite{LinZhang17} proved that they are polynomial in time and gave the dimensional estimate for the space of such solutions. The dimensional bound was improved by Colding and Minicozzi in \cite{ColdingMinicozzi19} recently. Note that the time analyticity of ancient solutions has been studied by \cite{Widder62, LinZhang17}. For a Riemannian manifold with Ricci curvature bounded below, Zhang \cite{Zhang19} proved that ancient solutions with exponential growth are time analytic. This result was improved to ancient solutions with double exponential growth in \cite{DongZhang19}.

\begin{theorem}[\cite{DongZhang19}]\label{time-analytic-manif2}
Let $M$ be a complete, $n$ dimensional noncompact Riemannian manifold such that the Ricci curvature satisfies $\Ric\geq -(n-1)K_0$ for some constant $K_0\geq 0$. Let $u$ be a solution of the heat equation $\partial_t u-\Delta u=0$ on $[-2,0]\times M$ of double exponential growth, namely
\[
|u(t,x)|\leq A_1e^{A_2d^2(x,0)}, \ \ \forall (t,x)\in [-2,0]\times M,
\]
where $A_1$ and $A_2$ are positive constants. Then $u=u(t,x)$ is analytic in $t\in [-1,0]$ with radius $r>0$ depending only on $n,$ $K_0$, and $A_2$. Moreover, we have
\[
u(t,x)=\sum_{j=0}^\infty a_j(x)\frac{t^j}{j!}
\]
with $\Delta a_j(x)=a_{j+1}(x)$, and
\[
|a_j(x)|\leq A_1A_3^{j+1}j^je^{2A_2d^2(x,0)}, \ j=0,1,2,\cdots,
\]
where $A_3$ is a positive constants depending only on $n, K_0$ and $A_2$.
\end{theorem}

For graphs, (discrete) harmonic functions of polynomial growth has been extensively studied by many authors, see e.g. \cite{Delmotte99, Kleiner10, HuaJost13, HuaJostLi11, HJL15,Hua19}. In this paper, we study the time analyticity of ancient solutions with exponential growth on graphs.

We recall the setting of weighted graphs. Let $G=(V,E)$ be a locally finite, simple, undirected graph with the set of vertices $V$ and the set of edges $E$.  Two vertices $x,y$ are called neighbours, denoted by $x\sim y$, if there is an edge connecting $x$ and $y$, i.e. $\{x,y\}\in E$. Let
\[\omega: E\rightarrow \mathbb (0,\infty),\ \{x,y\}\mapsto \omega_{xy}=\omega_{yx}\] be the edge weight function. We extend $\omega$ to $V\times V$ by setting $\omega_{xy}=0$ for any pair $(x,y)$ with $x\not\sim y.$ Let $$\mu:V\to (0,\infty), x\mapsto \mu_x$$ be the vertex weight function. We denote by $l^2(V,\mu)$ the space of $l^2$-summable functions on
$V$ with respect to the discrete measure $\mu$. We call the quadruple $G=(V, E, \mu,\omega)$ a \emph{weighted graph}.   

For a weighted graph $G=(V, E, \mu,\omega)$ and any function $f: V\rightarrow \mathbb R$, the Laplacian of $G$ is defined as
\[
\Delta f(x)=\sum_{y\in V}\frac{\omega_{xy}}{\mu_x}\left(f(y)-f(x)\right).
\]
For any vertex $x,$ we denote by
\[
\Deg(x):=\sum_{y\in V}\frac{\omega_{xy}}{\mu_x}
\]
the (weighted) degree of vertex $x$.
Note that the Laplacian $\Delta$ depends on the choice of weights $\omega$ and $\mu$. One can show that Laplacian is a bounded operator on $l^2(V,\mu)$ if and only if
\[
\sup_{x\in V}\Deg(x)<+\infty,
\]
see e.g. \cite{KellerLenz12}. The Laplacian is \emph{normalized} if $\Deg(x)\equiv 1$ for all $x\in V.$
We denote by $d$ the \emph{combinatorial distance}, that is, $d(x,y):=\inf\{n|x=z_0\sim \cdots \sim z_n=y\}$ and by
\[
B_R(x):=\{y\in V| d(y,x)\leq R\}
\]
the ball of radius $R>0$ centered at $x$. For any subset $K$ in $V,$ we denote by $(K)_1:=\{y\in V: \exists x\in K,\ \mathrm{s.t.}\ d(y,x)\leq 1\}$ the $1$-neighborhood of $K.$

A space-time function $u(x,t)$ is called an ancient solution on $V\times [-T,0],$ $T\in[0,\infty],$ if
\[
\frac{\partial}{\partial t}u(x,t)=\Delta u(x,t), \ \ \forall (x,t)\in V\times [-T,0].
\]
For ancient solutions of the heat equation on graphs, we prove the following result.
\begin{theorem}\label{main-result}
Let $G=(V, E, \mu,\omega)$ be a weighted graph and $p\in V.$ Let $u$ be an ancient solution of the heat equation $\partial_t u-\Delta u=0$ on $V\times[-T,0]$ satisfying
\begin{align}\begin{split}\label{exp-growth-cond}
|u(x,t)|\leq A_1e^{A_2d(x,p)\ln d(x,p)}, \ \ \forall (x,t)\in V\times [-T,0],
\end{split}\end{align}
where $A_1, A_2$ are positive constants.
Assume that for some $A_3\geq0,$
\begin{align}\begin{split}\label{deg-cond}
\Deg(x)\leq C d(x,p)^{A_3}, \ \ \forall x\in V.
\end{split}\end{align}
Suppose that $A_2+A_3\leq 1,$ then $u=u(x,t)$ is analytic in $t\in [-T,0]$ with analytic radius $r>0$ satisfying $r=+\infty$(resp. $r\geq\frac{1}{2eC}$) if $A_2+A_3<1$ (resp. $A_2+A_3=1$).

Moreover, for any $(x,t)\in V\times (-r,0]$, we have
\[
u(x,t)=\sum_{k=0}^\infty a_k(x)\frac{t^k}{k!}
\]
with $\Delta a_k(x)=a_{k+1}(x)$, and
\[
|a_k(x)|\leq A_1(2C)^ke^{(A_2+A_3)\left(k+d(x,p)\right)\ln \left(k+d(x,p)\right)}, \ k=0,1,2,\cdots.
\]
\end{theorem}
\begin{remark}
\begin{enumerate}
\item For Riemannian manifolds, the growth rate of the ancient solution is double exponential, i.e. $e^{A_2 d^2},$ see Theorem~\ref{time-analytic-manif2}, while in the above theorem we assume the growth rate is $e^{A_2d\log d}.$ Moreover, the factor $A_2$ in the growth rate is crucial for the time analyticity, for e.g. graphs with bounded degree ($A_3=0$), which is irrelevent for the continuous case.
\item Our estimate is sharp for graphs, see Section~\ref{sec:sharp}.
\end{enumerate}
\end{remark}
For Riemannian manifolds, the key estimate for the proof of Theorem~\ref{time-analytic-manif2} is the mean value property of the heat equation, which is unknown for graphs at the moment. To circumvent the problem, we take the advantages of the discrete nature of the Laplacian, e.g.
for a graph with bounded degree, $\Delta$ is a bounded operator on $\ell^\infty$ functions, and the support of the function $\Delta f$ is contained in the $1$-neighborhood of the support of $f,$ see Proposition~\ref{key-est}.

\begin{corollary}\label{time-analytic}
Let $G=(V, E, \mu,\omega)$ be a weighted graph and $p\in V$, and $u$ be a solution of the heat equation $\partial_t u-\Delta u=0$ on $V\times [-T,0]$ satisfying
\begin{align}\begin{split}
|u(x,t)|\leq A_1e^{A_2d(x,p)\ln d(x,p)}, \ \ \forall (x,t)\in V\times [-T,0],
\end{split}\end{align}
where $A_1$, $A_2$ are positive constants. Assume that for some constant $D>0$,
\[
\sup_{x\in V}\Deg(x)\leq D.
\]
Then $u=u(x,t)$ is time analytic with radius $r>0$ satisfying $r=+\infty$(resp. $r\geq \frac{1}{2eD}$) if $A_2<1$ (resp. $A_2=1$).
Moreover, for any $(x,t)\in V\times (-r,0]$, we have
\[
u(x,t)=\sum_{k=0}^\infty a_k(x)\frac{t^k}{k!}
\]
with $\Delta a_k(x)=a_{k+1}(x)$, and
\[
|a_k(x)|\leq A_1(2D)^k e^{A_2\left(k+d(x,p)\right)\ln\left(k+d(x,p)\right)}, \ k=0,1,2,\cdots.
\]
%where $A_4$ is a positive constants depending only on $A_1$.
\end{corollary}

\begin{corollary}\label{backward-heat-equation}
Let $G=(V, E, \mu,\omega)$ be a weighted graph and $p\in V$. Assume that for some constant $D>0$,
\[
\sup_{x\in V}\Deg(x)\leq D.
\]
Then the Cauchy problem for the backward heat equation
\begin{align}\begin{split}\label{backward-heat-equ}
\begin{cases}
\partial_t u+\Delta u=0,\\
u(x,0)=a(x)
\end{cases}
\end{split}\end{align}
has a solution on $[0,\delta)\times V$ for some constants $\delta>0$, $A_1>0$ and $0<A_2<1$ satisfying
\begin{equation}\label{linear-exponent}
|u(x,t)|\leq A_1e^{A_2d(x,p)\ln d(x,p)}
\end{equation}
if and only if
\begin{equation}\label{upp-bdd}
|\Delta^k a(x)|\leq A_3 \left(2D\right)^k e^{A_4\left(k+d(x,p)\right)\ln\left(k+ d(x,p)\right)}, k=0,1,2,\cdots.
\end{equation}
holds for some constant $A_3>0$ and $0<A_4<1$.
\end{corollary}
The paper is organized as follows: In next section, we prove Theorem \ref{main-result}. The last section is devoted to the example for the sharpness of the result.

\section{Proof of Theorem \ref{main-result}}\label{proof}

The following proposition is elementary, and hence we omit the proof here.
\begin{prop}\label{key-est}
Let $G=(V, E, \mu,\omega)$ be a weighted graph and $K\subset V$ be a subset. Then
\begin{enumerate}
\item $|\Delta f(x)|\leq 2\Deg(x) \sup\limits_{y\in B_1(x)}|f(y)|;$
\item $\sup\limits_{x\in K}|\Delta f(x)|\leq 2\sup\limits_{y\in K}\Deg(y)\sup\limits_{y\in (K)_1}|f(x)|.$
\end{enumerate}
\end{prop}

Now we prove the main result.
\begin{proof}[Proof of Theorem \ref{main-result}]
Given $R\geq 1$. Fix a vertex $x\in B_R(p)$. For $t_0\in [-T,0]$, by Talyor's theorem,
\be\label{talyor-extension}
u(x,t)-\sum_{j=1}^{k-1}\frac{\partial^j_tu(x,t_0)}{j!}(t-t_0)^j=\frac{\partial_t^ku(x,s)}{k!}(t-t_0)^k
\qe
holds for for any $t\in [-T,0]$ and $s\in (t_0, t)$ or $(t,t_0)$.
It suffices to prove that the right hand side of (\ref{talyor-extension}) tends to zero when $k$ tends to infinite for any $t\in [-T,0]$ satisfying $|t-t_0|<r$.

Since $\partial_t u=\Delta u$, and $\partial_t-\Delta$ is commutable with $\partial_t^j$, we obtain
\be\label{higorder-timeder-sol}
\partial_t(\partial_t^j u)=\Delta (\partial_t^j u).
\qe
For (\ref{higorder-timeder-sol}), by applying Proposition \ref{key-est},
\begin{align}\begin{split}\label{time-der-contr}
|\partial_t^{j+1} u(x,s)|
&\leq 2\Deg(x)\sup\limits_{y\in B_1(x)}|\partial_t^ju(y,s)|\\
&\leq 2^2\Deg(x)\sup_{y\in B_1(x)}\Deg(y) \sup\limits_{y\in B_2(x)}|\partial_t^{j-1}u(y,s)|\\
&\leq 2^{j+1}\left(\sup_{y\in B_j(x)}\Deg(y)\right)^{j+1}\sup\limits_{y\in B_{j+1}(x)}|u(y,s)|.
\end{split}\end{align}
%Without loss of generality, we assume $t_0<t$.
Let $\delta>0$ be a constant. Consider $t\in [t_0-\delta,t_0+\delta]\cap [-T,0]$.
%Without loss of generality, we assume that $s\in (t_0,t)$.
Let $\zeta=1- A_2-A_3$. Then $\zeta\geq 0$ since $A_2+A_3\leq 1$.
Hence, by (\ref{exp-growth-cond}), (\ref{deg-cond}) and (\ref{time-der-contr}),
\begin{align}\begin{split}\label{rem-term-est}
\Big|\frac{\partial_t^ku(x,s)}{k!}(t-t_0)^k\Big|
&\leq \frac{\delta^k}{k!}\sup\limits_{s\in (t_0,t)}|\partial_t^ku(x,s)|\\
&\leq \frac{\delta^k}{k!}2^k\Big(\sup_{y\in B_{k-1}(x)}\Deg(y)\Big)^k\sup\limits_{s\in (t_0,t),y\in B_k(x)}|u(y,s)|\\
&\leq \frac{(2\delta)^k}{k!}C^ke^{A_3k\ln (k+R)}\cdot A_1 e^{A_2(k+R)\ln(k+R)}\\
&\leq\frac{(2\delta C)^k}{k!}A_1 e^{(1-\zeta)(k+R)\ln (k+R)},
\end{split}\end{align}
where we used $d(y,p)\leq d(y,x)+d(x,p)\leq k+R.$

Stirling's approximation implies that there exists a number $k_0>0$ such that for any $k\geq k_0$,
\be\label{k!-est}
k!\geq \frac{k^k\sqrt{k}}{e^k}=\frac{e^{k\ln k+\frac{1}{2}\ln k}}{e^k}.
\qe
By Lagrange's mean value theorem, there exists $\xi\in (k,k+R)$ such that
%\[f'(x)=\ln x+1\leq \ln(k+R)+1\] holds for $x\leq k+R$.
\be\label{(k+R)ln(k+R)-est}
(k+R)\ln (k+R)-k\ln k=(\ln \xi+1)R\leq(\ln2+1+\ln k)R
\qe
holds for any $k\geq R$. Let
\[
Q=\frac{(2\delta C)^k}{k!}A_1 e^{(1-\zeta)(k+R)\ln (k+R)}.
\]

For $\zeta=0$, suppose that $\delta<\frac{1}{2eC}$. By (\ref{k!-est}) and (\ref{(k+R)ln(k+R)-est}), we have for $k\geq \max\{k_0,R\}$,
\begin{align}\begin{split}\label{rem-term-est1}
Q&\leq  A_1(2\delta Ce)^k e^{(\ln 2+1)R+\left(R-\frac{1}{2}\right)\ln k}\\
&=  A_1 e^{(\ln 2+1)R+\left(R-\frac{1}{2}\right)\ln k+k\ln(2\delta Ce)}\rightarrow 0,\ k\rightarrow \infty.
\end{split}\end{align}
This yields the result.

For $\zeta>0$, by (\ref{k!-est}) and (\ref{(k+R)ln(k+R)-est}), we have
\begin{align}\begin{split}\label{rem-term-est2 }
Q&\leq \frac{(2\delta C)^k}{k!} A_1e^{(1-\zeta)\left[k\ln k+(\ln 2+1)R+ R\ln k\right]}\\
 &\leq  A_1e^{-\zeta k\ln k-\frac{1}{2}\ln k+ k\ln(2\delta Ce)+(1-\zeta)\left[(\ln 2+1)R+R\ln k\right]}\\
 &=A_1e^{-\zeta k\ln k+\left[(1-\zeta)(\ln 2+1)R-\frac{1}{2}\ln k\right]+ k\ln(2\delta Ce)+(1-\zeta)R\ln k}\\
 &\leq A_1e^{-\frac{\zeta}{3}k\ln k}\rightarrow 0, k\rightarrow \infty,
\end{split}\end{align}
holds for
\[
k\geq \max\left\{k_0, R, e^{2(1-\zeta)(\ln 2+1)R}, (2\delta C e)^\frac{3}{\zeta}, \frac{3(1-\zeta)R}{\zeta}\right\}.
\]
This implies the result.

Combining (\ref{rem-term-est1}), (\ref{rem-term-est2 }) with (\ref{rem-term-est}), we obtain that the remaining term tends to zero when $k$ tends to infinite. Hence,
%for $(x,t)\in B_R(0)\times [t_0-\delta,t_0]$ with $t_0\in [-1,0]$, we have
\[
u(x,t)=\sum_{k=0}^{\infty}\partial^k_tu(x,t_0)\frac{(t-t_0)^k}{k!},
\]
that is, $u$ is time analytic with radius $r=+\infty$ (resp. $r\geq\frac{1}{2eC}$) if $A_2+A_3<1$ (resp. $A_2+A_3=1$).

For $t_0=0$, we write $a_k(x)=\partial^k_tu(x,0)$. Then
\[
\partial_tu(x,t)=\sum_{k=0}^{\infty}a_{k+1}(x)\frac{t^k}{k!}, \ \ \Delta u(x,t)=\sum_{k=0}^{\infty}\Delta a_{k}(x)\frac{t^k}{k!},
\]
the two series converge uniformly for $(x,t)\in B_R(p)\times (-\delta, 0]$.
Since $u(x,t)$ is an ancient solution of heat equation,
\[
\Delta a_k(x)=a_{k+1}(x).
\]
Moreover, as in the proof of (\ref{rem-term-est}), we get \[|a_k(x)|\leq A_1(2C)^ke^{(A_2+A_3)\left(k+d(x,p)\right)\ln \left(k+d(x,p)\right)}.\]
This completes the proof of the theorem.
\end{proof}

\begin{proof}[Proof of Corollary \ref{backward-heat-equation}]
On one hand, let $u(x,t)$ be a solution of (\ref{backward-heat-equ}) satisfying (\ref{linear-exponent}). Then $u(x,-t)$ is an ancient solution of the heat equation satisfying (\ref{linear-exponent}). By Corollary \ref{time-analytic}, $u(x,-t)$ is a time analytic solution and
\begin{align*}
u(x,-t)&=\sum_{k=0}^\infty\partial_{(-t)}^ku(x,0)\frac{(-t)^k}{k!}\\
       &=\sum_{k=0}^\infty\Delta^k a(x) \frac{(-t)^k}{k!}.
\end{align*}
Since $a_k(x)=\Delta^k a(x)$, then (\ref{upp-bdd}) follows.

On the other hand, suppose that (\ref{upp-bdd}) holds. Let
\[
q_k(x)=\Big|\Delta^k a(x)\frac{t^k}{k!}\Big|
\]
and $\zeta=1-A_2$. Since $A_2\leq 1$, $\zeta\geq 0$. Hence, for $t\in (-\delta,0]$ and $x\in B_R(p)$, by (\ref{k!-est}) and (\ref{(k+R)ln(k+R)-est}), we have
\begin{align*}
q_k(x)
&\leq  A_1 (2D)^ke^{A_2\left(k+d(x,p)\right)\ln\left(k+ d(x,p)\right)}\frac{\delta^k}{k!}\\
&\leq  A_1 \frac{(2D\delta )^k}{k!}e^{A_2\left(k+R\right)\ln\left(k+R\right)}\\
&\leq A_1 e^{-(1-A_2)k\ln k-\frac{1}{2}\ln k+k\ln(2\delta De)+A_2(\ln2+1)R+ A_2R\ln k }\\
&=A_1 e^{-\zeta k\ln k+\left[(1-\zeta)(\ln 2+1)R-\frac{1}{2}\ln k\right]+ k\ln(2\delta Ce)+(1-\zeta)R\ln k}.
\end{align*}
For $\zeta=0$ and $\delta<\frac{1}{2De}$,
\begin{align}\begin{split}\label{comp1}
q_k(x)
&\leq A_1 e^{\left[(\ln2+1)R-\frac{1}{2}\ln k\right]+k\ln(2\delta De)+R\ln k}\\
&\leq A_1 e^{k\ln(2\delta De)+R\sqrt{k}}\\
&\leq A_1 e^{\frac{1}{2}k\ln(2\delta De)}
\end{split}\end{align}
holds for
\[k\geq \max\{k_0, R, 2, \left(\frac{2R}{-\ln (2\delta De)}\right)^2\}.\]

For $\zeta>0$ and $\delta<+\infty$,  we obtain
\begin{align}\begin{split}\label{comp2}
q_k(x)
 &\leq A_1e^{-\zeta k\ln k+\left[(1-\zeta)(\ln 2+1)R-\frac{1}{2}\ln k\right]+ k\ln(2\delta Ce)}+(1-\zeta)R\ln k\\
 &\leq A_1e^{-\frac{\zeta}{3}k\ln k}
\end{split}\end{align}
holds for
\[
k\geq \max\left\{k_0, R, e^{2(1-\zeta)(\ln 2+1)R}, (2\delta De)^\frac{3}{\zeta}, \frac{3(1-\zeta)R}{\zeta}\right\}.
\]
Hence, by (\ref{comp1}) and (\ref{comp2}), one can show that the series $\sum_{k=0}^\infty\Delta^k a(x)\frac{t^k}{k!}$ is
converge absolutely and uniformly in $(-\delta,0]\times B_R(p)$ for any given $R\geq 1$. With the same arguments, we can obtain that
\[
\sum_{k=0}^\infty\Delta^{k+1} a(x)\frac{t^k}{k!}, \ \ \sum_{k=0}^\infty\Delta^k a(x)\frac{\partial_t t^k}{k!}
\]
are also converge absolutely and uniformly in $(-\delta,0]\times B_R(p)$.
Let
\[
v(x,t)=\sum_{k=0}^\infty\Delta^k a(x)\frac{t^k}{k!}.
\]
We claim that $v(x,t)$ is a solution of the heat equation for $t\in (-\delta,0]$.
Indeed,
\begin{align*}
\partial_t v(x,t)
&=\sum_{k=0}^\infty\Delta^k a(x)\frac{\partial_t t^k}{k!}\\
&=\sum_{k=0}^\infty \Delta^{k+1} a(x)\frac{t^k}{k!}
=\Delta v(x,t).
\end{align*}
Moreover, for $t\in (-\delta,0]$ and $\delta<\frac{1}{2De}$, we have
\begin{align*}
|v(x,t)|
&=\sum_{k=0}^\infty|\Delta^k a(x)|\frac{|t|^k}{k!}\\
&\leq \sum_{k=0}^\infty A_3 e^{A_4(k+d)\ln(k+d)}\frac{(2D\delta)^k}{k!}.
\end{align*}
Here we denote $d(x,p)$ by $d$ for convenience.
Since $A_4<1$, we choose a constant $\epsilon>0$ such that $A_4(1+2\epsilon)<1$.  Let
\[R_0=\max\left\{\frac{\ln(1+\epsilon)}{2\epsilon}, \frac{1}{\epsilon}k_0, \frac{1}{\epsilon}2^{\frac{1}{2\epsilon}}\right\}.\]
Since $f(x)=x\ln x$ is a convex function on $(0, +\infty)$,
\begin{align}\begin{split}\label{convex}
(k+d)\ln (k+d)\leq k\ln (2k)+d\ln(2d)=k\ln k+k\ln 2+d\ln(2d).
\end{split}\end{align}
Hence, by (\ref{k!-est}) and (\ref{convex}),
\begin{align*}
|v(x,t)|
&\leq \left(\sum_{k<\epsilon R_0}+\sum_{k\geq \epsilon R_0}\right) A_3 e^{A_4(k+d)\ln(k+d)}\frac{(2D\delta)^k}{k!}\\
&\leq A_3e^{A_4(\epsilon R_0+d)\ln(\epsilon R_0+d)}\sum_{k<\epsilon R_0}\frac{(2D\delta)^k}{k!} \\
&\ \ \ \ +A_3e^{A_4d\ln(2d)}\sum_{k\geq \epsilon R_0}e^{A_4k\ln k+A_4k\ln 2}\frac{(2D\delta e)^k}{k^k}\\
&:=I+II. 
\end{align*}
We divide it into two cases.

Case 1. $d(x,p)\leq R_0$. Since $A_4(1+2\epsilon)\leq 1$,
\begin{align*}
I&\leq A_3 e^{A_4(1+\epsilon)R_0\ln[(1+\epsilon)R_0]}\sum_{k<\epsilon R_0}\frac{(2D\delta)^k}{k!}\\
  &\leq A_3 e^{R_0\ln[(1+\epsilon)R_0]}\sum_{k<\epsilon R_0}\frac{1}{k!}=C_1(\epsilon)
\end{align*}
and
\begin{align*}
II
&\leq A_3e^{A_4R_0\ln (2R_0)}\sum_{k\geq\epsilon R_0} e^{-(1-A_4)k\ln k+ A_4k\ln2}(2D\delta e)^k\\
&\leq A_3e^{R_0\left(1+\ln R_0\right)}\sum_{k\geq\epsilon R_0}(2D\delta e)^k\\
&=C_2(\epsilon, D,\delta),
\end{align*}
where we used $1-A_4=2A_4\epsilon$ and $k\geq \epsilon R_0\geq 2^{\frac{1}{2\epsilon}}$ yields
\[-(1-A_4)\ln k+\ln2\leq 0.\]

Case 2. $d(x,p)> R_0$. By $2D\delta\leq 2D\delta e<1$ and $d>R_0\geq \frac{\ln(1+\epsilon)}{2\epsilon}$, we have
\begin{align*}
I&\leq A_3e^{A_4(1+\epsilon)d\ln(1+\epsilon)d}\sum_{k<\epsilon R_0}\frac{1}{k!}\\
 &\leq C_1(\epsilon)e^{A_4(1+2\epsilon)d\ln d}
\end{align*}
and
\begin{align*}
II
&\leq A_3 e^{A_4 d\ln(2d)}\sum_{k\geq\epsilon R_0}e^{-2A_4\epsilon k\ln k+A_4k\ln 2}(2D\delta e)^k\\
&\leq A_3 e^{A_4(1+2\epsilon ) d\ln d}\sum_{k\geq\epsilon R_0}(2D\delta e)^k\\
&=C_3(\epsilon, D, \delta) e^{A_4(1+2\epsilon ) d\ln d},
\end{align*}
where we used $1-A_4=2A_4\epsilon$, $k>\epsilon R_0\geq 2^{\frac{1}{2\epsilon}}$ implies that
\[
-(1-A_4)\ln k+\ln2\leq 0.
\]

Let $\tilde{A_4}=A_4(1+2\epsilon)$ and $\tilde{A_3}=\max\{C_1(\epsilon)+C_2(\epsilon, D, \delta), C_3(\epsilon, D, \delta)\}$. Combining Case 1 and Case 2, we obtain
\[
|v(x,t)|\leq I+II\leq \tilde{A_3}e^{\tilde{A_4} d(x,p)\ln d(x,p)}
\]
with $\tilde{A_4}<1$.
Thus $v(x,-t)$ is a solution to the Cauchy problem of the backward equation (\ref{backward-heat-equ}) and (\ref{linear-exponent}).
\end{proof}
\section{A sharpness example}\label{sec:sharp}
In this section, we construct an example to illustrate that the assumption $A_2+A_3\leq 1$ in Theorem \ref{main-result} is sharp.
We equip $\mathbb Z$ with a graph structure $(\mathbb Z, \mu_x, \omega_{xy})$ such that $\mu_x\equiv 1$ and for $x,y\in \mathbb Z$,
\[|x-y|=1 \Leftrightarrow x\sim y.\]
For the weighted graph $(\mathbb Z, \mu_x, \omega_{xy})$, Huang \cite[section 3]{Huang12} constructed a solution of heat equation with exponential growth. Making a slight modification, we construct an ancient solution (in fact entire solution) on $(\mathbb Z, \mu_x, \omega_{xy})$ with bounded Laplacian.
\begin{prop}
Let $(\mathbb Z, \mu_x, \omega_{xy})$ be a weighted graph with $\mu_x\equiv 1$ and $\omega_{xy}=1$ holds for $x\sim y$. For any $\epsilon>0$, there exists a constant $R_0=R_0(\epsilon)>0$ and an ancient solution of heat equation, $u(x,t)$, such that
\[
|u(x,t)|\leq Ae^{(1+\epsilon)x\ln x}
\]
holds for $x\geq R_0$ and some constant $A>0$, meantime, $u(x,t)$ is not time analytic.
\end{prop}
\begin{proof}
Let
\begin{align*}
g(t)=
\begin{cases}
\exp(-t^{-\beta}),\ \ &t>0,\\
0, & t\leq 0,
\end{cases}
\end{align*}
where $\beta$ is a constant. Clearly, all orders of derivatives of $g(t)$ goes to zero at zero.
For $0<T<+\infty$, we define a function $u(x,t)$ on $\mathbb Z\times [0,T]$ as follows:
\begin{align}\begin{split}\label{u(x,t)}
v(x,t)=
\begin{cases}
g(t), \ \ \ \ \ \ \ \ \ \ \ \ \ \ \ \ & x=0,\\
g(t)+\sum_{k=1}^\infty\frac{g^{(k)}(t)}{(2k)!}(x+k)\cdots (x+1)x\cdots (x-k+1),  & x\geq 1,\\
v(-x-1,t),  &x\leq -1.
\end{cases}
\end{split}\end{align}
Note that the function, $(x+k)\cdots (x+1)x\cdots (x-k+1)$, analogues to the power $x^{2k}$ in continuous case, is vanishing for all $k>|x|$. Recall the heat equation has a simple form as follows:
\[
\frac{d}{dt}v(x,t)+2v(x,t)-v(x-1,t)-v(x+1,t)=0.
\]
We can check that $u(x,t)$ solves the heat equation by direct computation.
Huang tells us that
\[
|g^{(k)}(t)|\leq k!\left(\frac{2k}{e\beta\theta^\beta}\right)^\frac{k}{\beta}
\]
for some $k>0$ and all $0<\theta<1$ small enough. Hence, by $|g(t)|\leq 1$,
\[
|v(x,t)|\leq 1+\sum_{k=1}^xb_k,
\]
where
\[b_k=\frac{k!}{(2k)!}(x+k)\cdots (x+1)n\cdots (x-k+1)\left(\frac{2k}{e\beta\theta^\beta}\right)^\frac{k}{\beta}.\]
By Lemma 3.2 in \cite{Huang12} we know that $b_k\leq b_{k+1}$. By Stirling's approximation, there exists a number $k_0>0$ such that
\be
k!\leq 2\sqrt{2\pi k}\left(\frac{k}{e}\right)^k
\qe
holds for $k\geq k_0$.
Set $C_0=\frac{1}{\theta}\left(\frac{2}{\beta}\right)^\frac{1}{\beta}$. We choose a constant $\theta$ such that $\theta$ satisfying $0<\theta<\min\left\{1, \left(\frac{2}{\beta}\right)^\frac{1}{\beta}\right\}$. This indicates that $C_0> 1$.
Then for all $x\geq \max\{k_0, e\}$, $t>0$,
\begin{align*}
|v(x,t)|&\leq 1+xb_x\\
        &= 1+x\frac{x!}{(2x)!}(2x)!\left(\frac{2x}{e\beta\theta^\beta}\right)^\frac{x}{\beta}\\
        &\leq 1+2x\sqrt{2\pi x}\left(\frac{x}{e}\right)^x \left[\left(\frac{x}{e}\right)^\frac{1}{\beta}C_0\right]^x\\
        &\leq 4\sqrt{2\pi}x^{(1+\frac{1}{\beta})x+\frac{3}{2}}e^{-(1+\frac{1}{\beta})x}C_0^x\\
        &\leq 4\sqrt{2\pi}e^{(1+\frac{1}{\beta})x\ln x+\left(\frac{3}{2}-(1+\frac{1}{\beta})+\ln C_0\right)x}.
%        &\leq 12e^{(1+\frac{1}{2\beta})x\ln x}.
        %&=1+ e^{(1+\frac{1}{\beta})x(\ln x-1)+\frac{3}{2}\ln x} \sqrt{2\pi}\left(\frac{2}{\beta\theta^\beta}\right)^\frac{x}{\beta}\\
        %&\leq Ae^{(1+\frac{1}{\beta})(1+\epsilon)x \ln x }
\end{align*}
Hence, for any $\epsilon>0$, $\beta>\max\{1, \frac{2}{\epsilon}\}$, set $R_0=\max\{e,k_0, e^{\frac{2}{\epsilon}\left[\frac{3}{2}-(1+\frac{1}{\beta})+\ln C_0\right]}\}$, we have
\[
|v(x,t)|\leq 4\sqrt{2\pi}e^{(1+\epsilon)x\ln x}
\]
holds for $x\geq R_0$.
Let
\[
u(x,t)=v(x,t+T).
\]
Then $u(x,t)$ is a heat solution on $V\times [-T,0]$.
If $u(x,t)$ is time analytic, then at $t_0=-T$,
\[
u(x,t)=\sum_{j=0}^\infty \partial^j_t u(x,-T)\frac{(t+T)^j}{j!}.
\]
Since $\partial^j_t u(x,-T)=\partial^j_t v(x,0)=0$ for any $j\geq 0$, $u(x,t)\equiv 0$ at $\mathbb Z\times [-T, -T+\delta)$, this contradicts to $u(x,t)=v(x,t+T)>0$ for $t>-T$ by (\ref{u(x,t)}).
Hence, $u(x,t)$ is not time analytic.
\end{proof}
\section*{Acknowledgements}
The authors would like to thank Professor Qi S. Zhang for his helpful suggestions and comments. B.H. is supported by NSFC, no.11831004 and no.11826031. L. W. is supported by NSFC, no.11671141 and China Postdoctoral Science Foundation, no. 2019M651332.

\bibliographystyle{alpha}
\bibliography{Ancient-solution}

\begin{thebibliography}{HJLJ11}

\bibitem[CM]{ColdingMinicozzi19}
Tobias~H. Colding and William P.~II Minicozzi.
\newblock Optimal bounds for ancient caloric functions.
\newblock arXiv:1902.01736.

\bibitem[Del99]{Delmotte99}
T.~Delmotte.
\newblock {Parabolic Harnack inequalities and estimates of Markov chains on
  graphs}.
\newblock {\em Revista Mathematica Iberoamericana}, 15:181--232, 1999.

\bibitem[DZ19]{DongZhang19}
Hongjie Dong and Qi~S. Zhang.
\newblock Time analyticity for the heat equation and navier-stokes equations.
\newblock 2019.
\newblock arXiv:1907.01687.

\bibitem[HJ13]{HuaJost13}
B.~Hua and J.~Jost.
\newblock $l^q$ harmonic functions on graphs.
\newblock {\em arXiv:1301.3403}, 2013.

\bibitem[HJL15]{HJL15}
B.~Hua, J.~Jost, and S.~Liu.
\newblock {Geometric analysis aspects of infinite semiplanar graphs with
  nonnegative curvature}.
\newblock {\em Journal f{\"u}r die reine und angewandte Mathematik}, 700:1--36,
  2015.

\bibitem[HJLJ11]{HuaJostLi11}
B.~Hua, J.~Jost, and X.~Li-Jost.
\newblock {Polynomial Growth Harmonic Functions on Finitely Generated Abelian
  Groups}.
\newblock {\em Ann. Grob. Anal. Geom. (published online), arXiv:1112.6284.},
  2011.

\bibitem[Hua]{Hua19}
Bobo Hua.
\newblock Dimensional bounds for ancient caloric functions on graphs.
\newblock arXiv:1903.02411.

\bibitem[Hua12]{Huang12}
Xueping Huang.
\newblock On uniqueness class for a heat equation on graphs.
\newblock {\em J. Math. Anal. Appl.}, 393(2):377--388, 2012.

\bibitem[KL12]{KellerLenz12}
M.~Keller and D.~Lenz.
\newblock {Dirichlet forms and stochastic completeness of graphs and
  subgraphs}.
\newblock {\em J. Reine Angew. Math.}, 666:189--223, 2012.

\bibitem[Kle10]{Kleiner10}
B.~Kleiner.
\newblock { A new proof of Gromov's theorem on groups of polynomial growth}.
\newblock {\em J. Amer. Math. Soc.}, 23(3):815--829, 2010.

\bibitem[LZ19]{LinZhang17}
Fanghua Lin and Q.~S. Zhang.
\newblock On ancient solutions of the heat equation.
\newblock {\em Comm. Pure Appl. Math.}, 72(9):2006--2028, 2019.

\bibitem[SZ06]{SoupletZhang06}
Philippe Souplet and Qi~S. Zhang.
\newblock Sharp gradient estimate and {Y}au's {L}iouville theorem for the heat
  equation on noncompact manifolds.
\newblock {\em Bull. London Math. Soc.}, 38(6):1045--1053, 2006.

\bibitem[Wid62]{Widder62}
D.~V. Widder.
\newblock Analytic solutions of the heat equation.
\newblock {\em Duke Math. J.}, 29:497--503, 1962.

\bibitem[Yau75]{Yau75}
S.~T. Yau.
\newblock {Harmonic functions on complete Riemannian manifolds}.
\newblock {\em Comm. Pure Appl. Math.}, 28:201--228, 1975.

\bibitem[Zha19]{Zhang19}
Qi~S. Zhang.
\newblock A note on time analyticity for ancient solutions of the heat
  equation.
\newblock 2019.
\newblock arXiv:1905.05845.

\end{thebibliography}

\end{document}